\documentclass[12pt]{amsart}
\usepackage{amscd}
\usepackage{amssymb}
\usepackage{a4wide}
\usepackage{amstext}
\usepackage{amsthm}
\usepackage{mathrsfs}
\usepackage{color}

\renewcommand{\phi}{\varphi}

\newcommand{\co}{\mathbb{C}}
\newcommand{\N}{\mathbb{N}}
\newcommand{\Z}{\mathbb{Z}}

\newcommand{\R}{\mathbb{R}}

\newcommand{\ff}{\mathcal{F}}

\newcommand{\mm}{\mathcal{M}}
\newcommand{\nn}{\mathcal{N}}
\newcommand{\cv}{c}
\newtheorem{Thm}{Theorem}[section]
\newtheorem{theorem}[Thm]{Theorem}
\newtheorem{lemma}[Thm]{Lemma}

\newtheorem{corollary}[Thm]{Corollary}

\begin{document}
\sloppy

\title[]{Complete interpolating sequences for the Gaussian shift-invariant space}
\author{Anton Baranov}
\address{Department of Mathematics and Mechanics\\ St. Petersburg State University\\
St. Petersburg, Russia}
\email{anton.d.baranov@gmail.com} 
\author{Yurii Belov}
\address{Department of Mathematics and Computer Science \\ St. Petersburg State University\\
St. Petersburg, Russia}
\email{j\_b\_juri\_belov@mail.ru}
\author{Karlheinz Gr\"ochenig}
\address{Faculty of Mathematics \\
University of Vienna \\
Oskar-Morgenstern-Platz 1 \\
A-1090 Vienna, Austria}
\email{karlheinz.groechenig@univie.ac.at}

%\date{}

\thanks{A.\,B. and Yu.\,B. were supported by the Ministry of Science and Higher Education of the Russian Federation, 
agreement No 075-15-2021-602, and by the Russian Foundation for Basic Research
grant 20-51-14001-ANF-a.
K.\,G.\ was supported in part by the  project P31887-N32  of the
Austrian Science Fund (FWF)}
\begin{abstract} {We give a full  description of complete interpolating sequences 
for the shift-invariant space generated by the Gaussian. As a
consequence, we  rederive
the known density conditions for sampling and interpolation.}
\end{abstract}

\maketitle

%\maketitle

\section{Main results}
\label{section1}

Consider the shift-invariant space of functions on $\R$ with Gaussian
generator $g(z) = e^{-az^2}$ for $a>0$ defined as
$$
V^2 =\bigg\{f(z) = \sum_{n\in \Z} c_ne^{-a (z-n)^2} : \ (c_n) \in\ell^2(\Z)    \bigg\}.
$$
We  consider the space $V^2$ as a subspace of $L^2(\R)$ with the
usual $L^2$-norm.

The space $V^2$ belongs to the general family of shift-invariant
spaces. Given a generator $g\in L^2(\R )$, such a space is defined as 
$$
V^2(g) =\{f(z) = \sum_{n\in \Z} c_ng(z-n) : \ (c_n) \in\ell^2(\Z)
\} \subseteq L^2(\R ). 
$$
The primary example is the classical
Paley--Wiener space $PW = \{f\in L^2(\R ): \mathrm{supp}\, \hat{f} \subseteq
[-1/2,1/2]\}$, which, by the sampling theorem of
Shannon--Whittaker--Kotelnikov, can be identified with the
shift-invariant space  $V^2(\frac{\sin \pi x}{\pi x})$.  In
signal processing~\cite{AG01} shift-invariant spaces are often taken as a
substitute for the Paley--Wiener space. A unifying feature of both
$V^2$ (Gaussian generator) and $PW$ (sinc-generator) is the fact that
both spaces can be viewed as spaces of entire functions by
interpreting the variable $z$ to be in $\co $. 

It is  easy to see  that the norm equivalence  $\|f\|_{L^2(\R)} \asymp
\|(c_n)\|_{\ell^2(\Z)}$  holds on $V^2$  with some absolute constants.  
In what follows it will be convenient for us to work with the second quantity and so we put 
$\|f\|_{V^2} :=  \|(c_n)\|_{\ell^2(\Z)}$, and we often identify $V^2$ with $\ell^2(\Z)$ 
via the mapping $(c_n)\mapsto f(z) =  \sum_{n\in \Z} c_ne^{-(z-n)^2}$.

A sequence $\Lambda = \{\lambda_n\}_{n\in\Z} \subset \R$ is said to be {\it sampling for $V^2$},
if 
$$
\sum_{n\in \Z} |f(\lambda_n)|^2 \asymp \|f\|^2_{V^2}, \qquad f\in V^2, 
$$ 
and we say that $\Lambda$ is {\it interpolating for $V^2$} if for
every $(a_n) \in \ell^2(\Z)$ 
there exists $f\in V^2$ such that $f(\lambda_n) = a_n$. If $\Lambda$ is both sampling and interpolating
 (or, equivalently, the solution of the interpolation problem is unique), we say that
$\Lambda$ is a {\it complete interpolating sequence}.  

Sampling and interpolation in the space $V^2$ and in more general shift-invariant spaces were studied
in \cite{grs1, grs2}. The main result can be formulated as follows. 
Recall that $\Lambda$ is said to be separated if $\inf_{\lambda,
  \lambda'\in \Lambda, \lambda\ne \lambda'} |\lambda-\lambda'|>0$, and
denote by $D^+({\Lambda})$ and $D^{-}({\Lambda})$ the  usual
upper and lower  Beurling densities 
$$
D^+(\Lambda) = \lim_{r\to\infty} \frac{\sup_{x\in\R} {\rm card}\,(\Lambda\cap [x, x+r])}{r},
\quad
D^-(\Lambda) = \lim_{r\to\infty} \frac{\inf_{x\in\R} {\rm card}\,(\Lambda\cap [x, x+r])}{r}.
$$

\begin{theorem} 
 \label{thm2}
 \begin{itemize}
 \item[(i)] Sufficiency:  Every separated sequence $\Lambda \subset \R$ 
 with $D^{-}({\Lambda})> 1$  is sampling for $V^2$.
Moreover, it contains a complete interpolating sequence.
 
 \item[(ii)] Necessity: If  $\Lambda$ is sampling for $V^2$, 
 then $\Lambda $  is a finite union of  separated subsequences and 
 $\Lambda$ contains a separated sequence 
 $\widetilde{\Lambda}$  such that  $D^{-}(\widetilde{\Lambda})\geq 1$.
 \end{itemize}
\end{theorem}
 
\begin{theorem} 
 \label{thm-int}
\begin{itemize}
 \item[(i)] Every separated sequence $\Lambda$ 
 with $D^{+}({\Lambda})< 1$
 is a set of interpolation for $V^2$. Moreover, it can be enlarged  to a complete interpolating sequence.

 \item[(ii)] If  $\Lambda$ is a set of interpolation for $V^2$, 
 then $\Lambda $ is  separated and  $D^{+}({\Lambda})\leq 1$.
 \end{itemize}
\end{theorem}

One of the main insights obtained in \cite{grs1, grs2} is the similarity of $V^2$ with the Paley--Wiener space $PW$
with respect to sampling and interpolation. In particular, for both
$V^2$ and $PW$ the same density conditions hold. For $PW$ these go
back to Beurling, Kahane, and Landau~\cite{beur89,Kah62,landau67,seip04}.  

Theorems~\ref{thm2} and \ref{thm-int} provide an almost
characterization of sampling sets and of interpolating sets, but they  leave
open the case of critical density. If a set is simultaneously
sampling and interpolating, then  $D(\Lambda ) = 1$. The case of the
critical density is much more subtle, because anything may
happen.  For $PW$ the problem of complete interpolating sequences  was solved  in~\cite{hnp}
and~\cite{OS02}, for $V^2$ it was open so far. 

Our main result is a complete and  explicit description of complete
interpolating sequences for $V^2$. Moreover, this characterization can
be extended to the larger  class of complex Gaussians
$$
g_\cv (z) = e^{-\cv z^2},  \qquad c=a+ib, \ \ a>0, \ \ b\in \R.
$$ 
We denote the space $V^2(g_\cv )$ by $V_c^2$. % , when necessary, but
% usually will omit the reference to the parameter $c$.

It is easy to see that for $f(z) = \sum_{n\in \Z} c_n e^{-\cv (z-n)^2}$ we still have 
$\|f\|_{L^2(\R)} \asymp \|(c_n)\|_{\ell^2(\Z)}$ and so 
we can define the norm in $V_\cv ^2 $ by $\|f\|_{V^2_\cv  } :=  \|(c_n)\|_{\ell^2(\Z)}$.

\begin{theorem}
\label{main1}  
Given $\cv \in\co$ with $a= \mathrm{Re}\, c >0$,  an increasing sequence $\Lambda
\subset\R $ is a complete interpolating sequence for $V^2_\cv $, if and only if $\Lambda$ 
is separated and there exists an enumeration $\Lambda=\{\lambda_n\}_{n\in \Z}$, $\lambda_n = n+\delta_n$, $n\in\Z$,
such that
   \begin{enumerate}
   \item[(a)] $(\delta_n)_{n\in\Z}\in \ell^\infty$;
   \item[(b)] there exists $N\geq 1$ and $\delta>0$ such that
$$
\sup_{n\in\Z} \frac{1}{N}\Big|\sum_{k=n+1}^{n+N}\delta_k\Big|\leq \delta<\frac{1}{2}.
$$
   \end{enumerate}
   \end{theorem}

Theorem~\ref{main1} is stronger than the results
   in~\cite{grs1}. Indeed, as a corollary, we will  deduce the  density results for sampling or interpolation
obtained previously in \cite{grs1} and extend them to the spaces generated by complex Gaussians.  

At least to us, Theorem~\ref{main1} is quite surprising, as it shows a
marked difference between $V^2_\cv $ and $PW$. The characterizing
conditions go back to Avdonin who showed that conditions (a) and (b)
are sufficient for complete interpolating sequences in $PW$. However,
they are far from necessary in $PW$. By contrast, in the
shift-invariant spaces $V^2_\cv $  these conditions provide a complete
characterization. 
   
To prove Theorem~\ref{main1}, we reduce the problem to the study of
similar problems in some Fock-type space of entire functions. After   
this reduction is done, one can apply the description of complete interpolating sequences in this space obtained in \cite{bdhk}.
\medskip

%%%%%%%%%%%%%%%%%%%%%%%%%%%%%%%%%%%%%%%%%%%%%%%%%%%%%%

\section{Unitary equivalence with a Fock-type space}
\label{section2} 

For a parameter $a>0$, we consider  the Fock-type space of entire
functions, sometimes referred to 
as a {\it small Fock space}, defined by 
$$
\ff =  \ff _a = \bigg\{F \ \text{entire}: \  \ \|F\|^2_{\ff _a} = 
\frac{1}{2 \sqrt{2\pi a } }\int_{\mathbb{C}} |F(w)|^2 e^{-\frac{1}{2a}(\log |w|)^2 }dm_2(w) < \infty \bigg\},
$$
where $m_2$ is the Lebesgue measure $\frac{1}{\pi} dxdy$.
A simple computation shows that for a function $F(w) = \sum_{n\ge 0} b_n w^n$ one has
$$
\|F\|^2_{\ff _a } = \sum_{n\ge 0} |b_n|^2e^{2a (n+1)^2}. 
$$

In what follows we assume $\cv \in\co $ to be fixed. %  and write $V^2$ in
% place of $V_\cv ^2$ 
% to simplify the notations.
To see the connection between $\ff _a$ and
$V_\cv ^2$, we  note that 
\begin{equation*}
  \label{eq:c1}
\sum_{n\in \Z} c_ne^{-\cv (z-n)^2} = e^{-\cv z^2} \sum_{n\in \Z} c_n
e^{-\cv n^2 } e^{2\cv nz}.  
\end{equation*}
  Introducing  a new variable $w=e^{2\cv z}$,  $V ^2_\cv $ becomes unitarily equivalent to the 
space of functions representable as
$$
\bigg\{ g(w) = \sum_{n\in \Z} d_n w^n: \ \  \|g\|^2 : = \sum_{n\in \Z}
|d_n|^2 e^{2a n^2} <\infty\bigg\}.
$$
This space consists of functions analytic in $\co\setminus\{0\}$. It
is  more convenient, however, to split this space 
into three parts, corresponding to positive and negative powers of
$w$, and the constant term.  Namely, for $f\in V^2_\cv $ we write
\begin{align}
f(z) = \sum_{n\in \Z} c_ne^{-\cv (z-n)^2} & = \sum_{n<0} c_ne^{-\cv (z-n)^2 } + c_0 e^{-\cv z^2 } 
+ \sum_{n>0} c_ne^{-\cv (z-n)^2 } \notag \\
& = f_-(z) + c_0 e^{-\cv z^2 }  + f_+(z). \label{eq:d2}
\end{align}

With respect to this decomposition $V^2_\cv $ splits into a direct %orthogonal
sum $V^2_\cv  = V^2_-\oplus e^{-\cv z^2 }\co \oplus V^2_+$,
which is, in terms of the coefficients, simply $\ell^2(\Z) = \ell^2(\N_-) \oplus\co\oplus\ell^2(\N)$.

Next, for $w=e^{2\cv z}$, set
$$
F_+(w) = \sum_{n=1}^\infty c_n e^{-\cv n^2} w^{n-1}, \qquad 
F_-(w) = \sum_{n=1}^\infty c_{-n} e^{-\cv n^2} w^{n-1} \, ,
$$
then in view of \eqref{eq:d2} we can write
$$
f(z) = e^{-\cv z^2} \big(w^{-1} F_-(w^{-1}) +c_0 +wF_+(w)\big)\, .
$$
 
It follows from the discussion above that $F_+, F_- \in \ff _a$ and
$$
\|F_+\|_{\ff _a} = \Big( \sum _{n>0} |c_n e^{-\cv n^2}|^2 e^{2an^2}
\Big)^{1/2}= \|c\|_2 =  \|f_+\|_{V^2_{\cv }}\, ,
$$
and $\|F_-\|_{\ff _a} = \|f_-\|_{V_\cv ^2}$. In addition every function $F_+, F_- \in \ff _a$ can appear in this representation and
$$
w F_+(w) = e^{\cv z^2} f_+(z)\Big|_{e^{2\cv z} = w}.
$$

This representation of elements in $V_\cv ^2$ makes it possible to
reduce the study of complete interpolating sequences in $V_\cv ^2$ to
the same problem in the Fock-type space $\ff _a $. 

Let us introduce the notions of sampling and interpolation for the
space $\ff _a$. Let $\varphi(z) = \frac{1}{4a} \log^2 |z|$. In the
following we drop the reference to the parameter $a$, when it is not
needed. 
We say that the sequence $\{w_n\}$ is sampling for $\ff $, if 
\begin{equation}
  \label{eq:c2}
\sum_{n} (1+|w_n|^2) e^{-2\phi(w_n)} |F(w_n)|^2 \asymp \|F\|^2_{\ff }, \qquad F\in \ff .  
\end{equation}
Analogously, $\{w_n\}$ is interpolating for $\ff$ if for every sequence $\{a_n\}$ satisfying
$\sum_{n} (1+|w_n|^2) e^{-2\phi(w_n)} |a_n|^2 <\infty$ there exists
$F\in\ff$ with $F(w_n) = a_n$. Recall from  \cite[Lemma~2.7]{bl} that the reproducing kernel
$k_w$ of $\ff $ has the norm
\begin{equation}
  \label{eq:c4}
  \|k_w\|_{\ff }^2 \asymp \frac{e^{2\phi (w)}}{1+|w|^2} \, .
\end{equation}
Thus the weights in the sampling inequality  are nothing  but
$\|k_{w_n}\|^{-2}_\ff$.   So the
norm equivalence 
\eqref{eq:c2} reads  as $\sum _n |\langle f, k_{w_n} \rangle _{\ff}|^2
\|k_{w_n}\|_{\ff} ^{-2} 
\asymp \|f\|_\ff ^2$. We may thus  formulate  the questions about sampling
and interpolation in the equivalent language of frames and Riesz sequences of 
normalized reproducing kernels $\tilde k_{w_n}
= k_{w_n}/\|k_{w_n}\|_{\ff}$ in $\ff$.
 Namely, $\{w_n\}$ is sampling for $\ff$ if and only if  $\{\tilde k_{w_n}\}$  is
a frame in $\ff$, while $\{w_n\}$ is interpolating if and only if $\{\tilde k_{w_n}\}$ is a Riesz sequence.
Finally, $\{w_n\}$ is a complete interpolating sequence if and only if $\{\tilde k_{w_n}\}$ is a Riesz basis in $\ff$.

We now  formulate the description of complete interpolating sequences
for $\ff _a$.
First it was shown in \cite{bl} that the sequence $\{e^{2an}\}_{n\ge
  1}$ is a complete interpolating sequence for $\ff _a$.
Then a full characterization of  complete interpolating sequences was
obtained in~~\cite{bdhk}. Remarkably, this characterization is formulated in terms of perturbations
of the complete interpolating sequence $\{e^{2an} \}_{ n\geq 1 }$,
somewhat in the spirit of the  theorems of Kadets and 
Avdonin for complex exponentials   (although  this condition
is only sufficient in $PW$).  
 % in the Paley--Wiener setting one can find only sufficient conditions in these terms).
Let $\{w_n\}_{n\ge 1}$ be a sequence such that 
$0<|w_n|\le |w_{n+1}|$ and let 
$$w_n = e^{2an} e^{\delta_n} e^{  i \theta_n}
$$ 
with $\delta_n, \theta_n \in\R$. In \cite{bdhk} the following result was proved.

\begin{theorem} \textup{\cite{bdhk}} 
\label{main2}  
A sequence $\{w_n\}_{n\ge 1}$ is a complete interpolating sequence for
$\ff _a$\footnote{Note that in ~\cite{bdhk} the normalization for the
small Fock space is different; to formulate the results from~\cite{bdhk}
one needs to set $\alpha = 1/(4a)$.},  
if and only if 
   \begin{enumerate}
   \item[(i)] $\{w_n\}$ is $\ff _a$-separated, i.e., there exists
     $\gamma >0$ such that $|w_m-w_n| \ge \gamma |w_n|$, $m\ne n$\textup;
   \item[(ii)] $(\delta_n)_{n\ge 1} \in \ell^\infty$\textup;
   \item[(iii)] there exists $N\geq 1$ and $\delta>0$ such that
$$
\sup_{n\ge 1} \frac{1}{N}\Big|\sum_{k=n+1}^{n+N}\delta_k\Big|\leq
\delta<a \, .
$$
   \end{enumerate}
   \end{theorem}

Note that the result does not depend on the arguments of $w_n$, but only on their moduli. For the description of
complete interpolating sequences for a more general class of Hilbert
spaces of entire functions see  
\cite{bms}.
\medskip

%%%%%%%%%%%%%%%%%%%%%%%%%%%%%%%%%%%%%%%%%%%%%%%%%%%%%%

\section{Proof of Theorem \ref{main1}: Sufficiency}
\label{section3} 

Assume that $\Lambda  \subset\R $ is separated and that 
there exists an enumeration $\Lambda=\{\lambda_m\}_{m\in \Z}$, $\lambda_m= m+\delta_m$, $m\in\Z$,
such that conditions (a) and (b) of Theorem~\ref{main1} are satisfied. We need to show that the mapping
\begin{equation}
\label{dzhi}
J:\  (c_n)_{n\in \Z}\mapsto (f(\lambda_m))_{m\in\Z}
\end{equation} 
is an isomorphism of $\ell^2(\Z)$ 
onto itself. Again we identify the sequence $(c_n)$ and the
corresponding function $f\in V^2_\cv$. 

Recall that $\phi (w) = \tfrac{1}{4a} (\log |w|)^2$, and put
$$
w_m = e^{2\cv  \lambda_m} = e^{2am} e^{2a \delta _m} e^{2ib(m+\delta
  _m)}  \, .
$$
The Avdonin-type condition (iii) of Theorem~\ref{main2} reads as
$\frac{1}{N}\Big|\sum_{k=n+1}^{n+N}(2a \delta_k)  \Big|\leq
\delta< a$, which amounts  precisely to the assumption (b) of
Theorem~\ref{main1}. Thus  each of the sequences $\{w_m\}_{m\ge 1}$ and 
$\{w_{-m}^{-1}\}_{m\ge 1}$ is a complete interpolating sequence for
$\ff _a$. 
\medskip

{\it Claim 1. The mapping $U_+: (c_n)_{n\ge 1} \mapsto (f_+(\lambda_m))_{m\ge 1}$ is an isomorphism 
from $\ell^2(\N)$ onto itself. } Indeed, since $\phi (w) = (\log
|w|)^2/(4a) = \lambda ^2$ whenever $w = e^{2\cv  \lambda}$,  we have  
\begin{equation}
\label{dro}
f_+(\lambda_m) = e^{-\cv \lambda_m^2} w_m F_+(w_m) =  
e^{-\phi(w_m)} e^{- ib\lambda_m^2} w_m F_+(w_m),
\end{equation}
and we conclude that  $(f_+(\lambda_m))_{m\ge 1} \in \ell^2(\N)$,
because $\{w_m\}$ is sampling,  and 
that every  sequence in  $\ell^2(\N)$ can be obtained in this way from
some function in $F_+$, because $\{w_m\}$ is interpolating. 
Thus $U_+$ is one-to-one and onto. Similarly 
the mapping  $U_-: (c_n)_{n\le -1} \mapsto (f_-(\lambda_m))_{m\le -1}$
is an isomorphism on  $\ell^2(\N_-)$. 
\medskip

{\it Claim 2. The mapping $K_+: (c_n)_{n\ge 1} \mapsto (f_+(\lambda_m))_{m\le - 1}$ is compact from
$\ell^2(\N)$ to $\ell^2(\N_-)$. } Since 
$$
f_+(\lambda_m) = \sum_{n\ge 1} c_n e^{-\cv (m+\delta_m -n)^2} \, ,
$$
the matrix of $K_+$ has the entries $e^{- \cv (m+\delta_m -n)^2} $,
$m<0,n>0$.  The boundedness of $\delta_m$ yields that  $\sum_{m<0,
  n>0} e^{-a(|m| +n -\delta_m)^2} <\infty$,  
thus $K_+$  is a Hilbert--Schmidt operator. Similarly, the operator 
 $K_-: (c_n)_{n\le -1} \mapsto (f_-(\lambda_m))_{m\ge 1}$ is compact from
$\ell^2(\N_-)$ to $\ell^2(\N)$. 
\medskip

Now we define the operator $U$  on $\ell^2(\Z) = \ell^2(\N_-) \oplus\co\oplus\ell^2(\N)$ by the formula 
\begin{equation}
\label{u}
U:\ \big((c_n)_{n\le -1};\, c_0;\, (c_n)_{n\ge 1}\big) \mapsto  
\big((f_-(\lambda_m))_{m\le -1};\, f(\lambda _0) ;\, (f_+(\lambda_m))_{m\ge 1}\big)
\end{equation}
By Claim 1 $U$ is an isomorphism from $\ell^2(\Z)$ onto itself. At the same time, by Claim 2, 
the map $J: \, (c_n)_{n\in \Z}\mapsto (f(\lambda_m))_{m\in\Z}$ differs from $U$
by a compact operator, since $f(\lambda_m) = f_+(\lambda_m) + c_0
e^{-\cv \lambda _m^2}+  f_-(\lambda_m)$. Thus, to prove that $J$ 
is invertible it is sufficient to show that $J$ is one-to-one, which
is usually easier. In our context  it means that $\Lambda$
is a uniqueness set for the space $V^2_\cv $.
\medskip

{\it Claim 3. $\Lambda$ is a uniqueness set for the space $V^2_\cv$. }
Assume that $f\in V^2_\cv$ and $f(\lambda_m) =0$ for all $m$. Then 
the function 
$$
F(w) =  w^{-1} F_-(w^{-1}) +c_0 +wF_+(w)
$$
vanishes  on the set $w_m = e^{2c \lambda_m}$.

Since $W_+ = \{w_m\}_{m\ge 1}$ is a complete interpolating sequence for
$\ff$, there exists a so-called {\it generating function} $G_+$ for
this sequence. This is  a function with  simple zeros 
exactly at $w_m$, $m\ge 1$,  and no other zeros, such that $G_+\notin
\ff$, but $\frac{G_+}{w-w_m} \in \ff$ for  all $m$.  
In fact,  the system $\big\{\frac{G_+}{G'_+(w_m)(w-w_m)}\}_{m\ge 1}$ is the biorthogonal system in $\ff$ to the system of reproducing
kernels $\{k_{w_m}\}_{m\ge 1}$. 

We will need several estimates for the generating functions. First, let $G_0$
be the generating function for the sequence $W_0 = \{e^{2a m}\}_{m\ge 1}$.
It was shown in \cite[Lemma 2.6]{bl} that
$$
 |G_0 (w)|\asymp e^{\varphi(w)}\frac{{\rm dist}\,(w,W_0)}{1+|w|^{3/2}}, \qquad w\in \co.
$$
Using standard estimates of lacunary canonical products it is easy to show (see \cite[Section 3, p. 1373]{bdhk})
that for the generating function $G_+$ of the perturbed sequence $W_+ = \{w_m\}_{m\ge 1}$ 
with $|w_m| = e^{2am+2a\delta_m}$ one has 
$$
 \frac{|G_+(w)|}{|G_0(w)|}\asymp  \exp\left(-2a\sum_{k=1}^{m}\delta_k\right) \frac{{\rm dist}\,(w,W_+)}{{\rm dist}\,(w, W_0)}
$$
whenever $e^{a(2m-1)} \le |w| \le e^{a(2m+1)}$.
By condition (b) of Theorem~\ref{main1}, for sufficiently large $m$ we have 
$\exp(-2a\sum_{k=1}^{m}\delta_k) \ge \exp(-2a\delta m) \asymp |w|^{-\delta}$, 
where $\delta<\frac{1}{2}$. Hence, 
\begin{equation}
\label{com1}
|G_+(w)| \gtrsim  \frac{{\rm dist}\, (w, W_+)}{(1+|w|)^{\frac{3}{2}+\delta}}\, e^{\phi(w)}.
\end{equation}  
 
Analogously, we define the function $G_-$, the generating function of the sequence 
$W_- = \{w_n^{-1}\}_{n\le -1}$. Consequently, the zero set of the function 
$(w-w_0) G_+(w) G_-(w^{-1})$ is precisely $\tilde{\Lambda}=\{w_n: n\in \Z \} $. Since  by
assumption $F$ vanishes also on $\tilde{\Lambda }$,  $F$ possesses the factorization 
$$
F(w) = (w-w_0) G_+(w) G_-(w^{-1}) H(w)
$$
for some function $H$ analytic in $\co\setminus\{0\}$. We will show that $H$ is an entire function
which tends to $0$ at infinity, and thus must be  identically zero.
Clearly, $w^{-1} F_-(w^{-1} ) \to 0$, as $|w| \to \infty $ and
$$
|F_+(w)| = |\langle F_+, k_w \rangle | \leq \|F_+\|_\ff \, \|k_w\|_\ff
\lesssim \frac{e^{\phi (w)}}{1+|w|} \, .
  $$
by \eqref{eq:c4}. Consequently, 
\begin{equation}
\label{com2}
|F(w)| \lesssim 1+ |wF_+(w)| \lesssim |w| \frac{e^{\phi(w)}}{1+|w|} \le e^{\phi(w)}, \qquad |w| \to\infty.
\end{equation}
% Here we used the fact that $\|k_w\|_{\ff} \asymp \frac{e^{\phi(w)}}{1+|w|}$.
Thus, using \eqref{com1}, \eqref{com2},  and the fact that
$G_-(w^{-1})\to G_-(0) \ne 0$, $|w|\to\infty$, 
we get 
\begin{align*}
|H(w)| & = \frac{|F(w)|}{|(w-w_0) G_+(w) G_-(w)|} \\
&   \lesssim  e^{\phi(w)} \cdot  \frac{(1+|w|)^{\frac{3}{2}+\delta}}{|w|{\rm dist}\, (w, W_+)}\, e^{-\phi(w)}
\lesssim
\frac{(1+|w|)^{\frac{1}{2}+\delta}}{{\rm dist}\, (w, W_+)}, \qquad |w|\to\infty. 
\end{align*}
It follows from the $\ff$-separation of $W_+$ and the inclusion 
$W_+\subset \R$  that there 
exists a sequence of radii $r_k\to \infty$ such that 
${\rm dist}\, (w, W_+) \ge \gamma |w|$ when $|w| = r_k$, with $\gamma >0$ independent on $k$.  Since $\delta<\frac{1}{2}$, 
we conclude that $\max_{|w|=r_k} |H(w)| \to 0$, as $k\to\infty$.  The
maximum principle implies that also  $H(w) \to 0$, as  $|w| \to \infty$.  

Analogously, replacing $w$ by $w^{-1}$ and using the  estimate  \eqref{com1}
for $G_-$, we conclude that $H(w) \to 0$ as $w\to 0$. Consequently the
singularity of $H$ at $0$ is removable, and $H $ is thus entire and
bounded. Thus  $H\equiv 0$ and so $F\equiv 0$, as was to be shown. 
\qed
\medskip

To conclude we formulate a simple,  sufficient condition for complete
interpolating sets in the style of Kadets, see e.g.~\cite{young}.

\begin{corollary} \label{corkad}
  Assume that $|\delta _n| \leq \delta < \frac{1}{2}$ for all $n\in
  \Z $. Then $\{n+\delta _n\}$ is a complete interpolating sequence
  for $V^2_\cv $. 
\end{corollary}

Note that in  the Paley--Wiener space $PW$ the corresponding result
holds with $\delta < 1/4$.  The constant $1/2$ is sharp, as the set
$\{n+1/2:n\in \Z \}$ is not a sampling sequence for $V^2_a$ with $a>0$~\cite{Jan93}. 
\medskip

%%%%%%%%%%%%%%%%%%%%%%%%%%%%%%%%%%%%%%%%%%%%%%%%%%%%%%

\section{Proof of Theorem \ref{main1}: necessity}
\label{section4} 

% {\color{blue} In the proof of necessity we will use the following statement about frames.

% \begin{proposition} 
% \label{exact} 
% Let $\{x_n\}_{n\ge 1}$ be a linearly independent system and a frame in a Hilbert space $H$. 
% If  $\{x_n\}_{n\ge 2}$ is not a frame, then  $\{x_n\}_{n\ge 1}$ is a Riesz basis in $H$. 
% \bigskip

% Another variant -- specifically for sampling and complete interpolating sequences in a space
% of analytic functions:
% \bigskip

% Let $\hh$ be reproducing kernel Hilbert space of analytic function in a domain $\Omega$
% and assume that $\hh$ has the division property.
% Assume that the sequence $\{w_n\}_{n\ge 1}$ of distinct points in $\Omega$
% is sampling for $\hh$ (i.e., $A\|f\|^2 \le \sum_n |f(w_n)|^2 \|k_{w_n}\|^{-2} \le  B \|f\|^2$ 
% for some $A,B>0$ and any $f\in\hh$), but $\{w_n\}_{n\ge 2}$ fails to be sampling. 
% Then $\{w_n\}_{n\ge 1}$ is a complete interpolating sequence for $\hh$.
% \end{proposition} 

% \begin{remark}
% \label{exact1}
% {\rm Obviously, it follows from Proposition \ref{exact} that if $W=\{w_n\}$ is a sampling sequence for $\hh$, but,
% for some finite subset $W_1$ of $W$, the sequence $W\setminus W_1$ is not sampling, 
% then there exists a subset $W_2\subset W_1$ such that $W \setminus W_2$ is a complete interpolating sequence. }
% \end{remark}
% }

For the proof of necessity we need to reverse the argument of the 
sufficiency part. Assume that $\Lambda$ is a complete interpolating sequence for $V^2_\cv$. This
means that  the operator $J$ given by \eqref{dzhi} is an isomorphism of
$\ell^2(\Z)$ onto  itself.

% \textbf{Step 1.}
Since a complete interpolating sequence for $V^2_\cv$  is always
separated,  % (see, e.g.,~\cite{})
the operators $K_+$ and $K_-$ defined in Claim 2 are compact, 
whence $U$ is a compact perturbation of the isomorphism $J$.

It follows that the kernel of $U$ has finite dimension and that the
range of $U$ has finite codimension:
\begin{equation}
  \label{eq:nov1}
{\rm dim}\,{\rm Ker}\, U <\infty \qquad \text{ and } \qquad {\rm
  codim}\,  {\rm  Range}\,
U < \infty \, .
  \end{equation}
%is a closed subspace of $\ell^2(\Z )$ of finite codimension.  

In the language of frame theory~\cite[Sec.~8.7]{heil} one speaks of the deficit
and the excess of the set of reproducing kernels. The condition
$\mathrm{Ker}\, U \neq \{0\}$ means that the reproducing kernels
associated to the sampling points $\lambda _n, n\in \Z$, are not
complete, but span a proper subspace (the \emph{deficit}). The finite codimension of the
Range of $U$ means that there are linear dependencies of the
reproducing kernels, in other words, too many functions for a Riesz
basis (the \emph{excess}). The finiteness condition of
\eqref{eq:nov1} implies that we can construct a Riesz basis of
reproducing kernels by adding and/or removing finitely many points.

In our context we will use an important property of the Fock space $\ff
$. Assume that  $F\in \ff$ satisfies  $F(w_1)= 0$ and  $w_2 \in \co , w_2 \neq w_1$. Then
the function $\tilde{F}(z) = \frac{z-w_2}{z-w_1} F(z)$ is again in
$\ff $. This property is called the \emph{division property} of $\ff $.

\medskip
\textbf{Step 1.} To start with the proof of necessity, assume that \eqref{eq:nov1}
holds.  Recall that $V_+^2 = \{f \in V^2_\cv:\, f=f_+\} \cong \ell^2(\N)$, and
consider  the restriction  $U_+ = U|_{V_+^2}$ of $U$ to the subspace 
$V_+^2$, given by $U_+(c_n)_{n\geq 1} = (f_+(\lambda _m))_{m\geq  1}$.    
From the properties of $U$ we  conclude that  the kernel
$\nn = \mathrm{Ker} \,  U_+$ 
in $V_+^2$ also is finite-dimensional and that ${\rm Range}\, U_+$ is a closed subspace of $\ell^2(\N)$ of finite
codimension.  Put $w_m =e^{2\cv \lambda_m}$, as before. If $U_+$ is an isomorphism
between $V^2_+$ and $\ell^2(\N)$, 
then,  by \eqref{dro}, $\{w_m\}_{m\ge 1}$ is a complete interpolating
sequence for $\ff$. If not, we distinguish two cases.

% \textbf{Step 2.} 
% We show that in the general case 
% $\{w_m\}_{m\ge 1}$ can be modified into  a complete interpolating
% sequence for $V^2_+$ by adding or deleting a finite set of points. 
\medskip

{\bf Case 1.} Assume that $U_+$ has a nontrivial kernel $\nn$. Then
$U_+|_{V_+^2 \ominus \nn}$ is an isomorphism on its image. Thus, 
$$
\sum_{m\ge 1} |f_+(\lambda_m)|^2 \asymp \|f_+\|_{V_+^2}^2, \qquad  f_+\in V_+^2 \ominus \nn.
$$
Note that $f\in \nn$ (i.e. $f_+(\lambda_m)=0$, $m\ge 1$) if and only if 
$F_+(w_m) = 0$, $m \ge 1$. 
By \eqref{dro}, we have
$$
\sum_{m\ge 1} (1+|w_m|^2) e^{-2\phi(w_m)} |F_+(w_m)|^2 \asymp \|F_+\|^2_{\ff}, \qquad F_+\in \ff\ominus \mm,
$$
where $\mm$ is the finite-dimensional subspace in $\ff$ which consists of all functions
vanishing on $\{w_m\}_{m\ge 1}$. 

We will need the following simple lemma. 

\begin{lemma} \label{lemadd}
Let $\mm = \{F\in \ff: \ F(w_m) = 0, m\ge 1\}$ 
and $M = {\rm dim}\, \mm <\infty$. Then there exists a set of points $\{\mu_k\}_{1\le k\le M}$ which is
a uniqueness set for $\mm$, while $\{\mu_k\}_{1\le k\le M-1}$ is a non-uniqueness set for $\mm$. 
\end{lemma}

\begin{proof}
We argue by induction on $M$. The base $M=1$ (i.e. $\mm = {\rm Lin}\,\{F_1\}$) 
is trivial since one can take as $\mu_1$ any point such that $F_1(\mu_1) \ne 0$. Assume that we can prove the statement
for $M-1$. Take any point $\mu_1 \in \co$ such that there is a function $F_0\in\mm$ with $F_0(\mu_1) \ne 0$
and consider $\mm_1 = \{F\in \mm: F(\mu_1) = 0\}$. It is obvious that ${\rm dim}\, \mm =M-1$. 
By the induction hypothesis there exist points $\mu_2, \dots, \mu_M$ which form a uniqueness set for $\mm_1$,
while $\{\mu_2, \dots, \mu_{M-1}\}$ is a non-uniqueness set for $\mm_1$. 
Then the points $\{\mu_k\}_{1\le k\le M}$ have the required property. 
\end{proof}

Let $\{\mu_k\}_{1\le k\le M}$ be the set constructed in the
lemma. Since $\mathcal{M}$ is finite-dimensional, it follows that  
$$
\sum_{k =1}^M (1+|\mu_k|^2) e^{-2\phi(\mu_k)} |F_+(\mu_k)|^2 \asymp
\|F_+\|^2_{\ff}, \qquad F_+\in \mm \, .
$$
Let  $\tilde W = \{w_m\}_{m\ge 1} \cup \{\mu_k\}_{1\le k\le M}$ the
enlarged set.
By construction, the sampling operator $F_+ \mapsto
\big( F_+(\tilde{w_k})\big)_{w_k\in \tilde{W}}$ is one-to-one and it has
closed range of finite codimension in $\ell^2(\tilde{W})$. Consequently,
by \cite[Thm. 8.29, (c)]{heil}, $\tilde{W}$ is a sampling set for $\ff$.

 By Lemma~\ref{lemadd} the set $\tilde{W} \setminus \{\mu_M\}$ is a
non-uniqueness set 
for $\ff$ and 
% , thus, is not sampling. By Proposition~\ref{exact} 
% $\tilde{W}$ is a complete interpolating sequence for $\ff$. }
% \medskip
% %Next we show that $\tilde{W}$ fails to be sampling after the removal
% %of an arbitrary point from $\tilde{W}$. By  this is clear
% %for the point $\mu _M$, %
therefore  there exists a function $H\in \ff $, such that $H(w_m)=
H(\mu _k)= 0$ for $m\in \N$ and $k=1, \dots , M-1$, but $H(\mu _M)
\neq 0$. We now use the division property of $\ff $ and show that
every set $\tilde{W} \setminus \{v\}$, where  $v \in \tilde{W}
\setminus  \{\mu _M\}$, is also a non-uniqueness set in $\ff$.  Consider $\tilde{H}(w) = \frac{w-\mu _M}{w-v} H(w)$. Then $\tilde{H} $ is  non-zero in $\ff $,  and $\tilde{H}$ vanishes
on $\tilde{W} \setminus \{v\}$, but  $\tilde{H}$ cannot vanish
at $v$, because $\tilde{W}$ is sampling on $\ff _+$.  

It follows that the frame of reproducing kernels $\{k_w\}_{w\in \tilde W}$ 
fails to be complete after removal of {\it any} of its members. 
We conclude that the normalized reproducing kernels corresponding to the sampling set 
$\tilde W$ form an {\it exact frame}, i.e., a frame which fails to remain a frame after removal of any of its vectors. 
By well-known results (see, e.g.,  \cite{heil,young}) the normalized reproducing kernels at $\tilde W$
form a Riesz basis in $\ff$, and, equivalently, $\tilde W$ is a complete  interpolating sequence for $\ff$.

\medskip
{\bf Case 2.} Assume now that ${\rm Ker}\,  U_+ = 0$, whence $U_+$ is an isomorphism onto its image
and so $\{w_m\}_{m\ge 1}$ is already sampling for $\ff$. Let ${\rm
  codim}\,{\rm Range}\, U_+ =M$. This means that the frame of
reproducing kernels $\{k_{w_m}\}_{m\ge 1}$ has finite excess, and therefore it becomes a
Riesz basis after the removal of finitely many elements,
see~\cite[Thm.~2.4]{Hol94} and \cite{heil}.

% We claim that the set 
% $\{w_m\}_{m\ge M+2}$ is a non-uniqueness set in $\ff$. 

% So assume that the orthogonal complement to 
% ${\rm Range}\, U_+$ in $\ell^2(\N)$ is spanned by the sequences $c_1, \dots, c_M$. 
% Then one can always find a nontrivial sequence $c = \{c^j\}_{j\ge 1}$ in $\ell^2(\N)$ 
% such that $c$ is orthogonal to $c_1, \dots, c_M$ and $c^j = 0$, $j\ge M+2$. It follows that $c$ 
% is in the range of $U_+$, i.e., $c=U_+ f_+$ for some $f_+\in V^2_+$. Thus, $f_+(\lambda_m) = 0$, $m\ge M+2$,
% and so the corresponding $F_+\in\ff$ vanishes on $\{w_m\}_{m\ge M+2}$. 

% {\color{blue}
% By Remark~\ref{exact1}, there exists $m_0\le M+1$ such that $\{w_m\}_{m\ge m_0}$ is a complete 
% interpolating sequence for $\ff$. } 
% \medskip

Repeating the argument for the negative part of $V^2_\cv$, we conclude that 
the set $\{w_m^{-1}\}_{m\le - 1}$ either can be enlarged  to a
complete interpolating set for $\ff$ or 
becomes a complete interpolating set after the removal of a finite set of points. 

\medskip

 \textbf{Step 2.}  It remains to show that after an appropriate enumeration the sequence $\{\lambda_m\}$ will satisfy conditions (a) and (b) 
of Theorem~\ref{main1}. We  note that both these conditions and also the property of being a complete interpolating sequence
for $\ff$ are stable with respect to moving a finite number of points
(without gluing them).
Formally, we remove a finite subset $F_1\subseteq \Lambda $ and add
a finite set $F_2$, such that $\mathrm{card}\, F_1 = \mathrm{card}\, F_2$ 
and $F_2 \cap \Lambda = \emptyset $. After
relabeling,  the new
set  $\Lambda ' = (\Lambda \setminus F_1) \cup F_2$ satisfies 
conditions (a) and (b), if and only if $\Lambda $ satisfies 
conditions (a) and (b).

In Step~1 we have shown that either  (i) there exists a finite set
$F_+ =  \{ \mu_j\} _{m=1, \dots , M}\}\subset \R$ % of cardinality $m_0$
such that  $\tilde{W}_+ = \{w_m\}_{m\ge 1} \cup \{e^{2c \mu
  _j}\}_{m=1, \dots , M}$ is a complete interpolating set
for $\ff$ or (ii)  we can remove a finite set $J_+ \subseteq
\{w_m\}_{m\ge 1}$, such that $\tilde{W}_+ = \{w_m\}_{m\ge 1} \setminus
J_+$ is a complete interpolating set for $\ff $. Using  the original
points $\lambda _m $ instead of $e^{2c\lambda _m}$ and calling the
removed points also $F_+$,  the  relabeled
sequence $\{\lambda _m\}_{m\geq 0} \cup F_+$ or $\{\lambda _m\}_{m\geq
  0} \setminus F_+$  then satisfies a one-sided version of (a) and (b) by
Theorem~\ref{main2}.  

% Similarly, if $\{w_m\}_{m\ge 1}$ becomes a complete interpolating set
% after removal of $m_0$ of its points (without loss of generality, the points $\lambda_m$, $m\le m_0$), 
% We add the set $F_+ = \{\mu_j\}$ to $\{\lambda_m\}_{m\ge 1}$
% and then relabel this set as $\{\lambda'_m\}_{m\ge 1}$ with $\lambda_m' = m +\delta'_m$. 
% Then the  relabeled sequence then satisfies a one-sided version of (a) and (b).
% we again can relabel the sequence $\{w_m\}_{m\ge m_0+1}$ to satisfy (a) and (b).  }

We apply the same procedure to the sequence $\{w_m^{-1}\}_{m<0}$ and
after adding  a finite set $F_- = \{ \nu _j\}$, we obtain a complete interpolating
sequence $\tilde{W}_- =  \{w_m^{-1}\}_{m<0} \cup \{e^{-2c\nu _j}\}$ or $\tilde{W}_-
=  \{w_m^{-1}\}_{m<0} \setminus J_-$.   Again a relabeling the
sequence $\{\lambda _m\}_{m< 0} \cup F_-$ or $\{\lambda _m\}_{m<
  0} \setminus F_-$ satisfies conditions (a) and (b). 

By taking the union of both sequences and returning to $\lambda _m,
\mu _j, \nu _j$ (instead of $w_m$),   we see  that there exist 
finite sets $F_1$ and $F_2$, satisfying $F_1 \subseteq \Lambda $ and
$F_2 \cap \Lambda = \emptyset $,   such that % either $\{\lambda_m\}_{m\in \Z} \cup J$ 
% or $\{\lambda_m\}_{m\in \Z} \setminus J$
$\Lambda ' =  (\Lambda \setminus F_1) \cup F_2$ can be enumerated to satisfy
(a) and (b). It is possible that $F_1=\emptyset $ (we  add points to both
$\{w_m\}_{m\geq 0}$ and to $\{w_m\}_{m<0}$) or $F_2 = \emptyset$ (we remove
points from both sets).

Hence, by sufficiency part 
of Theorem~\ref{main1}, $\Lambda '$ is a complete interpolating sequence for $V^2_\cv$. % either $\{\lambda_m\}_{m\in \Z} \cup J$ or $\{\lambda_m\}_{m\in \Z} \setminus J$ 
% is a complete interpolating sequence for $V^2_\cv$.
Note that $\Lambda = (\Lambda ' \setminus F_2) \cup F_1$. Let 
$\mathrm{card}\, F_1 = \mathrm{card}\, F_2 + \ell $. If $\ell >0$, we
have moved $\mathrm{card}\, F_2$ points  and added  $\ell $ additional
points. Then $\Lambda $ fails to be a complete interpolating sequence
for $V^2_{\cv }$ in contradiction to the assumption. If $\ell < 0$, we
have moved $\mathrm{card}\, F_1$ points and removed $\ell$ additional  points. Again $\Lambda $ fails
to be a complete interpolating sequence.

We conclude that $\mathrm{card}\, F_1 = \mathrm{card}\, F_2$ and
$\Lambda $ is obtained from $\Lambda '$ by moving finitely many
points. As observed above, since $\Lambda '$ satisfies conditions (a)
and (b), so must the original sequence $\Lambda $, as was to be
proved.  \qed

% However, by the hypothesis, $\{\lambda_m\}_{m\in \Z}$
% is itself a complete interpolating sequence for $V^2_\cv$. Hence, it fails
% to be complete interpolating after removing  any of its points
% or adding  new points.
% We conclude that the set $J$ is empty, and the
% original set $\Lambda $ satisfies conditions (a) and (b). 
%\qed
\medskip

% \begin{remark}
% {\rm In Step~2 one could argue alternatively that the frame
% $\{k_{w_n}\}_{n\geq 1}$ has finite excess, and therefore it becomes a
% Riesz basis after the removal of finitely many elements,
% see~\cite[Thm.~2.4]{Hol94} and \cite{heil}. }
% \end{remark}
\medskip

\section{Proof of the Density Results and Remarks}
\label{section5} 

Finally we indicate  how Theorems~\ref{thm2} and~\ref{thm-int} follow from the characterization
of complete interpolating sequences in $V^2_\cv$. We  use  a simple
``combinatorial'' argument from \cite{seip} in order to show that any
separated sequence $\Lambda$ with  
$D^-(\Lambda) >1$ contains a subsequence $\Lambda' =
\{m+\delta_m\}_{m\in\Z}$ with $\delta_m$ satisfying Avdonin-type
conditions (a) and (b) of Theorem~\ref{main1}. Analogously, any 
separated $\Lambda$ with $D^+(\Lambda) <1$ can be enhanced  to such a
set  $\Lambda'$. We refer to  \cite{bdhk} for the precise details, as
they are almost identical.  

The necessary density conditions for sampling and interpolation in a
shift-invariant space  are well-known, see, e.g.,
~\cite{AG01}. Let us show that they can also be deduced from the results for $\ff $ in
~\cite{bdhk}. 

  Let $\Lambda$ be an interpolating set for $V^2_\cv$. Then the mapping
$f\mapsto \{f(\lambda_m)\}_{m\ge 1}$ acts from $V^2_\cv$ onto $\ell^2(\mathbb{N})$.
By Claim 2 of the proof of sufficiency part in Theorem~\ref{main1} the map 
$f\mapsto  \{f_-(\lambda_m)\}_{m\ge 1}$ is a compact operator 
from $V^2_\cv$ to $\ell^2(\mathbb{N})$, whence the operator $f\mapsto  \{f_+(\lambda_m)\}_{m\ge 1}$
has the closed range of finite codimension. Then it is easy to see that for some $m_0 \ge 1$
the map $f\mapsto  \{f_+(\lambda_m)\}_{m\ge m_0}$ will be onto. Hence, the set 
$\{\lambda_m\}_{m\ge m_0}$ will be an interpolating set for $V^2_+$ and so $\{w_m\}_{m\ge m_0}$
will be interpolating for $\ff$. We conclude that  $D^+(\Lambda) \le 1$. 

Now assume that $\Lambda$ is a sampling set for $V^2_\cv$ and, in particular, for $V^2_+$. 
Since $\Lambda$ is a finite union of separated sets, the sequence $\{f_+(\lambda_m)\}_{m\le -1}$ decays very fast
and it is easy to show that the set $\{\lambda_m\}_{m\ge m_0}$ will be sampling for $V^2_+$.  
\medskip

\subsection*{Further remarks.} 

1. Let $V^p_\cv = \{f(z) = \sum_{n\in \Z} c_ne^{-\cv (z-n)^2} : \ (c_n)
\in\ell^p(\Z) \}$ be the subspace generated
by shifts of the Gaussian with $\ell ^p$ coefficients. This is a
closed subspace of $ L^p(\R )$. Again,  a sequence  $\Lambda
= \{\lambda _n: n\in \Z\}$ is called a sampling set  for $V_\cv ^p$, if 
$$
\sum_{n\in \Z} |f(\lambda_n)|^p \asymp \|f\|^p_{V^p_\cv } \asymp
\|c\|_p^p , \qquad f\in
V_\cv ^p, 
$$ 
and  {\it interpolating for $V^p_\cv $} if  any $(a_n) \in \ell^p(\Z)$  can
be interpolated by a function  $f\in V_\cv ^p$. Based on the general theory
of sampling in shift-invariant spaces with a nice generator, we can
assert that \emph{a set $\Lambda \subseteq \R $ is sampling for $V^2_\cv$,
  if and only if $\Lambda $ is sampling for some $V^{p_0}_\cv$,   if and
  only if $\Lambda $ is sampling for  $V^{p}_\cv $ for all $p, 1\leq p\leq
  \infty$}~\cite[Thm.~3.1]{grs1}. 
A similar statement holds for the interpolation
property. Theorem~\ref{main1} therefore provides a characterization
for complete interpolating sequences for all spaces $V^p_\cv $. 
\medskip

2. \emph{Sign retrieval.} The sign retrieval problem asks whether a
real-valued function $f$ in some function space is uniquely determined
by its unsigned values $|f(\lambda )|$ on some set $\Lambda $. For
$V^2_a$ with a real-valued Gaussian generator,  Theorem~\ref{main1} implies the following result on sign
retrieval. 

\begin{corollary} 
  Let $\Lambda = \{n/2 + \delta _n: n\in \Z \}$ be
  uniformly separated  such that $\{\delta _n\} $ satisfies the Avdonin-type
  conditions $(a)$ and $(b)$ with $\delta < 1/4$. Then every $f\in V^2_a
  $ is uniquely determined by its unsigned values $\{|f(n/2 + \delta _n)|: n\in
  \Z \}$.
\end{corollary}

The proof is the same as  in~\cite{Gr20a}, we simply
replace sampling sets of density $D^-(\Lambda ) >2$ by   a
complete interpolating sequence for a dilated version of $V^2_a$. Thus
sign retrieval is possible even at the critical density. 
\medskip

3. It would be interesting to obtain similar results for complete
interpolating sequences in shift-invariant spaces with other
generators. Even a sharp version of Corollary~\ref{corkad} would have
important consequences  for Gabor frames.

\end{document}